\documentclass{article}
\usepackage{maa-monthly}
\usepackage[final]
{showkeys}
\usepackage{graphicx}
\usepackage{enumerate}
\usepackage{multirow}
\usepackage[skip=40pt,font=large]{caption}
\usepackage{framed}

\usepackage[framemethod=tikz]{mdframed}

\usepackage{trimclip}

\usepackage{needspace}

\usepackage{esint}

\usepackage{hyperref}

\theoremstyle{plain}
\newtheorem{theorem}{Theorem}

\newtheorem{corollary}[theorem]
{Corollary}

{Claim}

\theoremstyle{definition}

\newtheorem{example}
{Example}
\newtheorem{remark}
{Remark}

\theoremstyle{definition} 
\newtheorem*{remark*}{Remark}

\renewcommand{\skip}[1]{}

\newcommand{\N}{\mathbb{N}}

\newcommand{\E}{\operatorname{\mathsf{E}}}
\renewcommand{\P}{\operatorname{\mathsf{P}}}
\newcommand{\Var}{\operatorname{\mathsf{Var}}}

\newcommand{\1}{
{\mathbf{1}}}

\newcommand{\A}{\mathbf{U}}
\newcommand{\B}{\mathbf{V}}

\newcommand{\PP}{\mathcal{P}}
\newcommand{\U}{\mathcal{U}}
\newcommand{\AAA}{\mathfrak{A}}
\newcommand{\f}{\mathfrak{f}}

\newcommand{\si}{\sigma}

\newcommand{\Om}{\Omega}

\newcommand\ep{\varepsilon}

\makeatletter
\DeclareRobustCommand{\lasymp}{\lg@asymp{<}}
\DeclareRobustCommand{\gasymp}{\lg@asymp{>}}

\newcommand{\under@asymp}[1]{\clipbox{0pt 0pt 0pt {0.5\height}}{$\m@th#1\asymp$}}
\newcommand{\lg@asymp}[1]{\mathrel{\mathpalette\lg@asymp@{#1}}}
\newcommand{\lg@asymp@}[2]{%
  \vcenter{%
    \offinterlineskip
    \m@th
    \ialign{%
      \hfil##\hfil\cr
      $#1#2$\cr
      \under@asymp{#1}\cr
    }%
  }%
}
\makeatother

\newcommand{\F}{\mathcal F}
\newcommand{\cc}{\mathsf c}

\makeatletter

\long\def\@maketablecaption#1#2{%
    \rule[.1\baselineskip]{0pt}{\baselineskip}
    \small\textbf{#1.}\enspace #2\strut
    \par
  \vskip2pt}
\makeatother  

\begin{document}

\title{Equi-dependence implying independence 
}
\markright{
\quad\quad\quad\quad\quad\ \  Equi-dependence implying independence }
\author{Iosif Pinelis \\ 
Michigan Technological University \\ 
Houghton, Michigan USA \\ 
Email: ipinelis@mtu.edu 
}


%
%
%
%
%

\maketitle

\begin{abstract} 
It is shown that, if an event $A$ has the same conditional probability in each trial in an infinite sequence of Bernoulli trials, then $A$ is independent of each trial. More general results are actually established. 

\noindent {\bf Keywords:} Conditional probability; independence. 
\end{abstract}








We have a probability space $(\Om,\F,\P)$ and independent events $B_1,B_2,\dots$ each of probability $p\in(0,1)$: 
\begin{equation}\label{eq:=p}
	\P(B_i)=p
\end{equation}
for all $i\in\N=\{1,2,\dots\}$. As usual, here the term \emph{event} means a member of the underlying $\sigma$-algebra $\F$. 

\begin{theorem}\label{th:1}
Suppose that $A$ is 
an event such that
\begin{equation}\label{eq:=q}
	\P(A|B_i)=q
\end{equation}
for some $q\in(0,1)$ and all $i\in\N$. 
Then $A$ is independent of $B_i$ for each $i\in\N$. 
\end{theorem}

\begin{remark}\label{rem:}
Note that the independence of $A$ and $B_i$ means that the conditional probability $\P(A|B_i)$ is the same as the unconditional probability $\P(A)$. So, the conclusion of Theorem~\ref{th:1} can be rewritten simply as $\P(A)=q$. 
\end{remark}

Under the conditions of Theorem~\ref{th:1}, the conclusion that $A$ is independent of the entire sequence $(B_1,B_2,\dots)$ of events would not hold in general. For a counterexample, suppose that $p=1/2$ and let $A:=(B_1\cap B_2)\cup(B_1^\cc\cap B_2^\cc)$, so that \eqref{eq:=q} holds for all $i\in\N$ with $q=1/2$; as usual, here $^\cc$ denotes the complement operation. 

On the other hand, condition \eqref{eq:=q} can be significantly relaxed, so that, in view of Remark~\ref{rem:}, we have the following generalization of Theorem~\ref{th:1}: 

\begin{theorem}\label{th:2}
Suppose that condition \eqref{eq:=p} holds for all $i\in\N$, and for some $q\in(0,1)$ we have 
\begin{equation*}
	\frac1n\,\sum_{i\in[n]} \P(A|B_i)\to q
\end{equation*}
as $n\to\infty$, where, as usual, $[n]:=\{1,\dots,n\}$. Then $\P(A)=q$. 
\end{theorem}

Theorem~\ref{th:2} can be further generalized, as follows. For any 
subsets $K$ and $L$ of $\N$, let 
\begin{equation}\label{eq:B^{K,L}}
	B^{K,L}:=\Big(\bigcap_{k\in K}B_k\Big)\cap\Big(\bigcap_{l\in L}B_l^\cc\Big); 
\end{equation}
of course, the intersection of any empty family of events is 
$\Om$.  

For any two nonnegative integers $s$ and $t$ and any integer $n\ge s+t$, let 
\begin{equation*}
	\PP_{n,s,t}:=\Big\{(S,T)\colon S\in\binom{[n]}s,\ T\in\binom{[n]}t,\ S\cap T=\emptyset \Big\}
\end{equation*}
and 
\begin{equation*}
	N_{n,s,t}:=|\PP_{n,s,t}|=\binom n{s,t,n-s-t}=\frac{n!}{s!\,t!\,(n-s-t)!}, 
\end{equation*}
where $|\cdot|$ denotes the cardinality, 
and, as usual, for any set $E$, $\binom Ek$ denotes the set of all subsets of $E$ of cardinality $k$. 

\begin{theorem}\label{th:3}
Take any two nonnegative integers $s$ and $t$. Suppose that condition \eqref{eq:=p} holds for all $i\in\N$, and for some $q\in(0,1)$
\begin{equation}\label{eq:s,t,to q}
	\frac1{N_{n,s,t}}\,\sum_{(S,T)\in\PP_{n,s,t}}\P(A|B^{S,T})\to q
\end{equation}
as $n\to\infty$. 
Then $\P(A)=q$. 
\end{theorem}

Theorem~\ref{th:2} is a special case of Theorem~\ref{th:3}, with $s=1$ and $t=0$.  

The following immediate corollary of Theorem~\ref{th:3} is a generalization of Theorem~\ref{th:1}. 

\begin{corollary}\label{cor:}
Take any two nonnegative integers $s$ and $t$. Let 
\begin{equation*}
	\PP_{\infty,s,t}:=\Big\{(S,T)\colon S\in\binom\N s,\ T\in\binom\N t,\ S\cap T=\emptyset \Big\}. 
\end{equation*} 
Suppose that condition \eqref{eq:=p} holds for all $i\in\N$ and 
\begin{equation}\label{eq:S,T,q}
	\P(A|B^{S,T})=q
\end{equation}
for some $q\in(0,1)$ and all $(S,T)\in\PP_{\infty,s,t}$.   
Then $A$ is independent of all events of the form $B^{S,T}$ for $(S,T)\in\PP_{\infty,s,t}$. 
\end{corollary}

These results, and especially Corollary~\ref{cor:}, are reminiscent of the Hewitt--Savage zero-one law (HSL) \cite{hewitt-savage}, where a certain condition on finite subsequences of an infinite sequence of independent random variables (r.v.'s) implies that the probability of a certain event is either $0$ or $1$ (so that this event is independent of all events). However, Corollary~\ref{cor:} differs significantly from the HSL. 

On the other hand, 
condition \eqref{eq:s,t,to q} permits certain perturbations of condition \eqref{eq:S,T,q}. 
Thus,  
Theorem~\ref{th:3} may be viewed as a statement of stability in Corollary~\ref{cor:}. Similarly, Theorem~\ref{th:2} may be viewed as a statement of stability in Theorem~\ref{th:1}. 


One may conjecture that, given the conditions of 
Corollary~\ref{cor:}, it is possible to conclude more: that the event $A$ is independent of the entire sequence $(B_i)_{i\in\N}$ of events. 
The following example, which generalizes the comment made after 
Remark~\ref{rem:}, provides a rather strong negation of this conjecture. 

\begin{example}\label{ex:}
Suppose that $p=1/2$. 
For any $k\in\N_0:=\N\cup\{0\}$, let $D_k$ denote the event that an even number of events $B_1,\dots,B_k$ occurs. Then for any $k\in\N$ 
\begin{equation*}
	\P(D_k)-\P(D_k^\cc)=\sum_{j=0}^k \binom kj (-1)^j \frac1{2^k}=0,
\end{equation*}
so that $\P(D_k)=1/2$. 

Take now any $n\in\N$ and let $A:=D_n$. 
For any assertion $\AAA$, let $\1(\AAA)$ denote the truth value of $\AAA$, so that $\1(\AAA)=1$ if $\AAA$ is true and $\1(\AAA)=0$ if $\AAA$ is false. 
For any nonnegative integers $s$ and $t$ such that $s+t<n$ and any $(S,T)\in\PP_{n,s,t}$,
\begin{multline*}
	\P(A\cap B^{S,T})=\P(B^{S,T})\big(\1(s\text{ is even})\,\P(D_{n-s-t})
	+\1(s\text{ is odd})\,\P(D_{n-s-t}^\cc)\big) \\ 
	=\Big(\frac12\Big)^{s+t}\,\frac12=\Big(\frac12\Big)^{s+t+1}. 
\end{multline*}
That is, $\P(A|B^{S,T})=1/2$;  
here it was essential that $s+t<n$, because $\P(D_0)=1$. 

Let us now extend the conclusion $\P(A|B^{S,T})=1/2$ for $(S,T)\in\PP_{n,s,t}$ to any $(S,T)\in\PP_{\infty,s,t}$.
For any nonnegative integers $s$ and $t$ such that $s+t<n$ and any $(S,T)\in\PP_{\infty,s,t}$, letting $I:=S\cap[n]$ and $J:=T\cap[n]$, and recalling the independence of the $B_i$'s, we have 
\begin{multline*}
	\P(A\cap B^{S,T})=\P(A\cap B^{I,J})\P(B^{S\setminus I,T\setminus J})
	=\Big(\frac12\Big)^{|I|+|J|+1}\Big(\frac12\Big)^{|S|-|I|+|T|-|J|} \\ 
	=\Big(\frac12\Big)^{|S|+|T|+1}=\frac12\,\P(B^{S,T}),
\end{multline*}
so that 
indeed $\P(A|B^{S,T})=1/2$. 
Thus, for any nonnegative integers $s$ and $t$, choosing any natural $n > s + t$ and letting $A := D_n$ ensures that condition \eqref{eq:S,T,q} holds, with $q = 1/2$, for all $(S,T)\in\PP_{\infty,s,t}$.
%
Note also that, by Corollary~\ref{cor:} or just because $A=D_n$, we have $\P(A)=1/2$.
However, event $A=D_n$ is of course not independent of the entire sequence $(B_i)_{i\in\N}$, because $A$ belongs to the $\si$-algebra generated by events $B_1,\dots,B_n$ and $\P(A)=1/2\notin\{0,1\}$. \qed 
\end{example}


\bigskip


The following proofs of Theorems~\ref{th:1} and \ref{th:3} were suggested, in the main, by anonymous referees. The original proof of Theorem~\ref{th:3} was longer and more complicated. 
Of course, Theorem~\ref{th:1} is immediately implied by Theorem~\ref{th:3}. However, the proof of Theorem~\ref{th:1} is significantly simpler and shorter than that of Theorem~\ref{th:3}, and it already contains, in simpler forms, the essential ideas. Therefore, we begin the necessary proofs with the following. 

\begin{proof}[Proof of Theorem~\ref{th:1}] Let 
\begin{equation}\label{eq:M_n}
	M_n:=\frac1n\,\sum_{i\in[n]}\1_{B_i}, 
\end{equation}
where $\1_B$ denotes the indicator of an event $B$, so that $\1_B=1$ if $B$ occurs and $\1_B=0$ if $B$ does not occur. 
By \eqref{eq:=p}, $\E M_n = p$ and, by the independence, $\Var M_n = \break 
p(1-p)/n$. 
So, $\E|M_n-p|^2=\Var M_n\to0$ (as $n\to\infty$) and hence \break 
$\E|M_n-p|\to0$. It follows that  
$$|\E\1_A M_n-\E\1_A p|=|\E\1_A(M_n-p)|\le\E\1_A|M_n-p|
\le \E|M_n-p|\to0,$$ 
which implies  
$\E\1_A M_n\to \E\1_A p=\P(A)p$. 
On the other hand, by \eqref{eq:=p} and \eqref{eq:=q}, $\E\1_A M_n=pq$. 
Thus, $\P(A)p=pq$ and hence $\P(A)=q$. It remains to recall \break Remark~\ref{rem:}.
\end{proof}

\begin{proof}[Proof of Theorem~\ref{th:3}] In this more general setting, definition \eqref{eq:M_n} in the proof of Theorem~\ref{th:1} is naturally generalized as follows: 
\begin{equation*}
	M_n:=\frac1{N_{n,s,t}}\,\sum_{(S,T)\in\PP_{n,s,t}}\1_{B^{S,T}}. 
\end{equation*}
Then 
\begin{equation}\label{eq:EM}
	\E M_n=r:=p^s(1-p)^t,
\end{equation}
since, in view of \eqref{eq:B^{K,L}},  
\begin{equation}\label{eq:r}
	\E\1_{B^{S,T}}=\P(B^{S,T})=p^{|S|}(1-p)^{|T|}=p^s(1-p)^t  
\end{equation}
for all $(S,T)\in\PP_{n,s,t}$. Next, 
\begin{equation*}
	\Var M_n=\frac1{N_{n,s,t}^2}\,\sum_{(S,T)\in\PP_{n,s,t}}\sum_{(Q,R)\in\PP_{n,s,t}}
	c_{S,T;Q;R},
\end{equation*}
where $c_{S,T;Q;R}$ is the covariance of $\1_{B^{S,T}}$ and $\1_{B^{Q,R}}$. In view of the independence of the $B_i$'s, $c_{S,T;Q;R}=0$ if $(S\cup T)\cap(Q\cup R)=\emptyset$. So, $c_{S,T;Q;R}\ne0$ only if $S\cap Q\ne\emptyset$ or $T\cap R\ne\emptyset$ or $S\cap R\ne\emptyset$ or $T\cap Q\ne\emptyset$. 
Also, clearly $c_{S,T;Q;R}\le1$ for all $(S,T)$ and $(Q,R)$ in $\PP_{n,s,t}$. Therefore, 
\begin{equation}\label{eq:varN}
	\Var M_n\le\pi_{n,s,s}+\pi_{n,t,t}+\pi_{n,s,t}+\pi_{n,t,s},
\end{equation}
where 
\begin{equation*}
	\pi_{n,u,v}:=\P(\A\cap\B\ne\emptyset) 
\end{equation*}
and 
$\A$ and $\B$ are independent random subsets of $[n]$ 
such that $\A$ is picked uniformly at random from all subsets of $[n]$ with $u$ elements and $\B$ is picked uniformly at random from all subsets of $[n]$ with $v$ elements. 

Let us now bound $\pi_{n,u,v}$, with $u$ and $v$ in $\{0,\dots,n\}$. 
By conditioning, we
may fix an arbitrary realization $U$ of $\A$, where $\A$ and $\B$ are as above. The chance of a single uniformly random element of $[n]$ to be in $U$ is $|U|/n=u/n$. So, by the union bound, the chance of at least one element of $\B$ to be in $U$ is at most $|\B|u/n=uv/n$. Thus, 
%
\begin{equation}\label{eq:le uv/n}
	\pi_{n,u,v}\le\frac{uv}n. 
\end{equation}

The conclusion of the proof of Theorem~\ref{th:3} is quite similar to that of Theorem~\ref{th:1}. 
It follows from \eqref{eq:EM}, \eqref{eq:varN}, and \eqref{eq:le uv/n} that $\E|M_n-r|^2=\Var M_n\to0$ and hence 
$$|\E\1_A M_n-\E\1_A r|=|\E\1_A(M_n-r)|\le\E\1_A|M_n-r|
\le \E|M_n-r|\to0,$$ 
so that 
\begin{equation*}
	\E\1_A M_n\to \P(A)r. 
\end{equation*}
On the other hand, by \eqref{eq:s,t,to q} and \eqref{eq:r}, 
\begin{equation*}
\E\1_A M_n=
	\frac1{N_{n,s,t}}\,\sum_{(S,T)\in\PP_{n,s,t}}\P(B^{S,T})\P(A|B^{S,T})\to rq. 
\end{equation*}
We conclude that $\P(A)r=rq$ and hence $\P(A)=q$. 
\end{proof}

\appendix

\section{Appendix} In the following counterexample, kindly communicated to me by Qiyuan Gu, $p$ and $q$ can be any numbers in the interval $(0,1)$---in contrast with \break 
Example~\ref{ex:}, where $p=q=1/2$. 


Take any nonnegative integers $s$ and $t$, and then take any natural $n>s+t$. 
Let $X_1,X_2,\dots$ be independent Bernoulli($p$) random variables (r.v.'s) and let 
$B_i := \{X_i = 1\}$ for all $i\in\N$. Let $H := \prod_{j=1}^n (X_j - p)$. Let  also a r.v.\ $U$ be uniformly distributed over the interval $(0,1)$ and independent of the sequence $(B_i)_{i\in\N}$. 
Take any $\ep\in(0 ,\min(q, 1-q))$, and let $F := q + \ep H$ and $A := \{U \le F\}$. Since $|H| < 1$, we have $0 < F < 1$. Also $\E H = 0$, so $\P(A) = q$.

Now take any 
$(S,T)$ in $\mathcal P_{\infty,s,t}$. Since $|S \cup T| < n$, there is some $j\in[n]\setminus(S \cup T)$. By independence, $\E H 1_{B^{S,T}} = 0$, because the factor $X_j - p$ in $H$ has mean zero and $j\notin S \cup T$. Therefore, $\P(A \cap B^{S,T}) = \E F 1_{B^{S,T}} = q \P(B^{S,T})$, and hence $\P(A | B^{S,T}) = q$. Thus, condition \eqref{eq:S,T,q} holds. 

On the other hand, $A$ is not independent of the sequence $(B_i)_{i\in\N}$. Indeed, for $C := B_1 \cap \dots \cap B_n$, we have $H = (1-p)^n$ on $C$, and therefore $\P(A \cap C) - \P(A)\P(C)  = \ep p^n (1-p)^n  > 0$. \qed

\bigskip

After this example, there remained the question of whether there is 
%
a similar example with $A$ in the $\sigma$-algebra generated by the sequence $(B_i)_{i\in\N}$. 

ChatGPT provided the following nice and simple positive answer to this latter question. Indeed, take any $p$ and $q$ in $(0,1)$, and any nonnegative integers $s$ and $t$. 
Note that the (distribution of the) r.v.\ 
\begin{equation*}
	W:=\sum_{i\in\N}2^{-i}\1_{B_i}
\end{equation*}
is non-atomic, because any real number has at most two binary representations, and the probability of each representation is $\le\max(p,1-p)^n$ for all natural $n$. Letting $F$ denote the c.d.f.\ of $W$, we get the r.v.\ $V:=F(W)$, which is measurable with respect to the $\sigma$-algebra generated by the sequence $(B_i)_{i\in\N}$ and uniformly distributed over the interval $[0,1]$. 

Let now $n:=s+t+1$ and partition $\N$ into infinite sets $I_1,\dots,I_n$. With the above construction of $V$ in mind, construct independent r.v.'s $V_1,\dots,V_n$ such that for each $j\in[n]$ the r.v.\ $V_j$ is measurable with respect to the $\sigma$-algebra generated by the sequence $(B_i)_{i\in I_j}$ and uniformly distributed over the interval $[0,1]$. 

Note first the shift property of the uniform distribution $\U[0,1]$ over $[0,1]$: 
if \break 
$U\sim\U[0,1]$ 
then $\f(x+U)\sim\U[0,1]$ for any real $x$, where $\f$ denotes the fractional part. 
It follows that $Z:=\f(V_1+\dots+V_n)\sim\U[0,1]$. Also, $Z$ 
is measurable with respect to the $\sigma$-algebra generated by the sequence $(B_i)_{i\in\N}$. 
Moreover, since 
$n>s+t$, for each $(S,T)\in\PP_{\infty,s,t}$ there is some $j\in[n]$
such that $I_j\cap(S\cup T)=\emptyset$, and hence, again by the shift property,   
the conditional distribution of $Z$ given 
$B^{S,T}$ 
is uniform over $[0,1]$.

So, letting $A:=\{Z<q\}$, we see that \eqref{eq:S,T,q} holds for all $(S,T)\in\PP_{\infty,s,t}$. However, $A$ is in the $\sigma$-algebra generated by the sequence $(B_i)_{i\in\N}$ and is not independent of $(B_i)_{i\in\N}$, because $\P(A)=q\notin\{0,1\}$. \qed



%
%
%
\bibliographystyle{vancouver}

\def\cprime{$'$} \def\polhk#1{\setbox0=\hbox{#1}{\ooalign{\hidewidth
  \lower1.5ex\hbox{`}\hidewidth\crcr\unhbox0}}}
  \def\polhk#1{\setbox0=\hbox{#1}{\ooalign{\hidewidth
  \lower1.5ex\hbox{`}\hidewidth\crcr\unhbox0}}}
  \def\polhk#1{\setbox0=\hbox{#1}{\ooalign{\hidewidth
  \lower1.5ex\hbox{`}\hidewidth\crcr\unhbox0}}} \def\cprime{$'$}
  \def\polhk#1{\setbox0=\hbox{#1}{\ooalign{\hidewidth
  \lower1.5ex\hbox{`}\hidewidth\crcr\unhbox0}}} \def\cprime{$'$}
  \def\polhk#1{\setbox0=\hbox{#1}{\ooalign{\hidewidth
  \lower1.5ex\hbox{`}\hidewidth\crcr\unhbox0}}} \def\cprime{$'$}
  \def\cprime{$'$}

\nopagebreak
\nopagebreak

\end{document}